\documentclass[a4paper,11pt]{article}

\usepackage[utf8]{inputenc}
\usepackage[english]{babel}
\usepackage{amsmath,amssymb,amsthm,,amsbsy,bbm} 
\usepackage{hyperref,xcolor,fullpage}
\usepackage{dsfont}
\usepackage{faktor}
\usepackage{todonotes}
\usepackage{tikz}
\usepackage{xfrac}
\usepackage{float}
\tikzset{
	pics/carc/.style args={#1:#2:#3}{
		code={
			\draw[pic actions] (#1:#3) arc(#1:#2:#3);
		}
	}
}

\makeatletter \let\@fnsymbol\@arabic \makeatother 

\allowdisplaybreaks 
\usepackage[numbers]{natbib} 

\newtheorem{theorem}{Theorem}[section]
\newtheorem{lemma}[theorem]{Lemma}
\newtheorem{example}[theorem]{Example}

\newtheorem{corollary}[theorem]{Corollary}

\newtheorem{conjecture}[theorem]{Conjecture}

\newtheorem{proposition}[theorem]{Proposition}

\makeatletter
\let\save@mathaccent\mathaccent
\newcommand*\if@single[3]{
	\setbox0\hbox{${\mathaccent"0362{#1}}^H$}%
	\setbox2\hbox{${\mathaccent"0362{\kern0pt#1}}^H$}%
	\ifdim\ht0=\ht2 #3\else #2\fi }
\newcommand*\rel@kern[1]{\kern#1\dimexpr\macc@kerna}
\newcommand*\widebar[1]{\@ifnextchar^{{\wide@bar{#1}{0}}}{\wide@bar{#1}{1}}}
\newcommand*\wide@bar[2]{\if@single{#1}{\wide@bar@{#1}{#2}{1}}{\wide@bar@{#1}{#2}{2}}}
\newcommand*\wide@bar@[3]{
	\begingroup
	\def\mathaccent##1##2{
		\let\mathaccent\save@mathaccent
		\if#32 \let\macc@nucleus\first@char \fi
		\setbox\z@\hbox{$\macc@style{\macc@nucleus}_{}$}
		\setbox\tw@\hbox{$\macc@style{\macc@nucleus}{}_{}$}
		\dimen@\wd\tw@ \advance\dimen@-\wd\z@ \divide\dimen@ 3 \@tempdima\wd\tw@ \advance\@tempdima-\scriptspace \divide\@tempdima 10 \advance\dimen@-\@tempdima \ifdim\dimen@>\z@ \dimen@0pt \fi \rel@kern{0.6}\kern-\dimen@
		\if#31 \overline{\rel@kern{-0.6}\kern\dimen@\macc@nucleus\rel@kern{0.4}\kern\dimen@} \advance\dimen@0.4\dimexpr\macc@kerna \let\final@kern#2 \ifdim\dimen@<\z@ \let\final@kern1 \fi
		\if \final@kern1 \kern-\dimen@ \fi
		\else \overline{\rel@kern{-0.6}\kern\dimen@#1} \fi }
	\macc@depth\@ne	\let\math@bgroup\@empty \let\math@egroup\macc@set@skewchar 	\mathsurround\z@ \frozen@everymath{\mathgroup\macc@group\relax} 	 \macc@set@skewchar\relax \let\mathaccentV\macc@nested@a	\if#31 \macc@nested@a\relax111{#1} \else \def\gobble@till@marker##1\endmarker{} \futurelet\first@char\gobble@till@marker#1\endmarker \ifcat\noexpand\first@char A\else \def\first@char{} \fi \macc@nested@a\relax111{\first@char} \fi
	\endgroup }
\makeatother

\usepackage{mathrsfs}
\DeclareMathAlphabet{\mathpzc}{OT1}{pzc}{m}{it}
\usepackage{xifthen}
\newcommand{\fp}[1][]{ \ifthenelse{\isempty{#1}}{\mathpzc{p}}{\mathpzc{p}(#1)} }

\newcommand{\mS}{\mathcal{S}}

\newcommand{\mA}{\mathcal{A}}

\newcommand{\mR}{\mathcal{R}}

\newcommand{\NN}{\mathbb{N}}
\newcommand{\ZZ}{\mathbb{Z}}

\definecolor{darkred}{cmyk}{.3,.9,.80,.2}

\title{A step beyond Freiman's theorem \\ for set addition modulo a prime}
\author{
Pablo Candela\thanks{Universidad Aut{\'o}noma de Madrid, Ciudad Universitaria de Cantoblanco, Madrid 28049, Spain, {\tt pablo.candela@uam.es}.} \and
Oriol Serra\thanks{Universitat Polit{\'e}cnica de Catalunya, Department of Mathematics, 08034 Barcelona, Spain, {\tt oriol.serra@upc.edu}. Supported by the Spanish Ministerio de Econom{\'i}a y Competitividad projects MTM2014-54745-P and MDM-2014-0445.} \and
Christoph Spiegel\thanks{Universitat Polit{\'e}cnica de Catalunya and Barcelona Graduate School of Mathematics, Department of Mathematics, Edificio Omega, 08034 Barcelona, Spain, {\tt christoph.spiegel@upc.edu}. Supported by the Spanish Ministerio de Econom{\'i}a y Competitividad FPI grant under the project MTM2014-54745-P, the project MTM2017-82166-P and the María de Maetzu research grant MDM-2014-0445.}
}

\begin{document}
\maketitle

\begin{abstract}
	Freiman's 2.4-Theorem states that any set $A \subset \ZZ_p$ satisfying $|2A| \leq 2.4|A| - 3 $ and $|A| < p/35$ can be covered by an arithmetic progression of length at most $|2A| - |A| + 1$. A more general result of Green and Ruzsa implies that this covering property holds for any set satisfying $|2A| \leq 3|A| - 4$ as long as the rather strong density requirement $|A| < p/10^{215}$ is satisfied. We present a version of this statement that allows for sets satisfying $|2A| \leq 2.48|A| - 7$ with the more modest density requirement of $|A| < p/10^{10}$.
\end{abstract}

\section{Introduction}

Given a set $A \subset G$ in some additive group $G$, we define its \emph{sumset} as 
\begin{equation}
	A + A = \{ a+a' : a,a' \in A \} \subset G.
\end{equation} 
We will often denote this sumset by $2A$, which should not be confused with the dilate $2 \cdot A = \{2a : a \in A\}$. When dealing with inverse questions in additive combinatorics, one is typically interested in understanding the structure of a set $A$ for which only some additive property is known, e.g. that the so-called \emph{doubling} $|2A|/|A|$ is small. One of the most important results in this area is Freiman's Theorem, which states that any finite set of integers can be efficiently covered by a generalized arithmetic progression, where the size and the dimension of the  progression depend only on the doubling. The bounds for this result were later improved and the ambient group generalized to many different contexts, see for example~\cite{Chang02,GR07,Ruzsa94,Sanders08,Schoen11}.

In the more specific case of finite sets of integers $A \subset \ZZ$ with \emph{very} small sumsets, that is $|2A| \leq 3|A| - 4$, Freiman showed that in fact $A$ can be covered by a normal (i.e. 1-dimensional) arithmetic progression of length at most $|2A| - |A| + 1$. This result is easily seen to be tight. An equivalent statement in the cyclic group $\ZZ_p$, where $p$ is a prime, is widely believed to hold as well, assuming certain modest restrictions regarding the cardinality of $A$ with respect to $p$. However, such a statement has turned out to be more difficult to prove.

It was Freiman himself who first showed that the same covering property holds for any set $A \subset \ZZ_p$ satisfying $|2A| \leq 2.4|A| - 3$ and $|A| < p/35$, see~\cite{Freiman61}. R{\o}dseth~\cite{Rodseth06} later showed that the density requirement can be weakened to $p/10.7$. A more general result of Green and Ruzsa~\cite{GR06} immediately gives the same conclusion for all sets satisfying $|2A| \leq 3|A| - 4$ as long as they also satisfy the rather strong density requirement $|A| < p/10^{215}$. The second author and Z{\'e}mor obtained a result with the same covering conclusion and no  restrictions regarding the size of $|A|$ itself, but assuming that $|2A| \leq (2+\varepsilon)|A|$ with $\varepsilon<10^{-4}$, see~\cite{SZ09}.

We present a version of this statement that improves upon the constant $2.4$ present in the results of Freiman and R\o dseth, without requiring quite as strong a density condition as in the result of Green and Ruzsa.

\begin{theorem} \label{thm:main}
Let $A \subset \ZZ_p$ satisfy $|2A| \leq 2.48|A| - 7$ and $|A| < p/10^{10}$. Then there is an arithmetic progression $P\subset \ZZ_p$ such that $A\subset P$ and $|P|\leq |2A| - |A| + 1$.
\end{theorem}

Similar to both the result of Freiman and that of Green and Ruzsa, the proof of this statement uses a Fourier-analytic rectification argument that allows one to transplant a significant part of the set into the integers, where the corresponding covering result can be applied. Unlike the result of Freiman however, we will allow the doubling of that part to go past the $3|A|-4$ barrier in the integers. To do so, we will use a result of Freiman and Deshoulliers to prove a covering result for sets of integers with doubling slightly above that barrier. This also implies that, unlike in the original approach, we have to take the additive dimension of our sets into consideration.


We believe that the ideas in this paper are capable of yielding significantly better constants than the ones obtained here. The biggest obstacle in improving either the density requirement or the constant $2.48$ will be the relatively poor covering given by Proposition~\ref{prop:4k-8_partial}. A conjecture of Freiman (see for example~\cite{FS16}) claims that in fact one should be able to replace the constant $10^9$ with $4$ in that proposition, resulting in a significant improvement of the density requirement of our main statement. So far only very little has been proven in that direction. Freiman himself solved the case $3|A|-3$, and Jin~\cite{Jin07} obtained a result in the case $(3+\varepsilon)|A|$ for some undetermined $\varepsilon > 0$.

\bigskip

\noindent {\bf Outline. } In \emph{Section~\ref{sec:preliminaries}} we will introduce some tools required in order to prove Theorem~\ref{thm:main}. We will first state and prove a covering result for sets of integers having doubling slightly above the $3|A|-4$ threshold. We will then give an overview of some well-established rectification principles. Using these tools we will then prove Theorem~\ref{thm:main} in \emph{Section~\ref{sec:proof}}. We make some concluding remarks in \emph{Section~\ref{sec:remarks}}.

\section{Preliminaries} \label{sec:preliminaries}

Let us formally define some common notions and concepts. We say that a set $A \subset \ZZ$ is in \emph{normal form} if $A \subset \NN_0$, $0 \in A$ and $\gcd(A) = 1$. Note that one can easily put any finite set $A \subset \ZZ$ into normal form without affecting its cardinality or additive properties, by setting $A' = (A-\min(A))/\gcd(A-\min(A))$. If a set $A$ is covered by an arithmetic progression of length $k$ then it follows that the normal form $A'$ of that set satisfies $\max(A') \leq k-1$.

Let $A$ and $B$ be two subsets of some (not necessarily identical) abelian groups. A bijection $f: A \to B$ is said to be a \emph{Freiman isomorphism of order $k$}, or \emph{$F_k$-isomorphism} for short, if for any elements $a_1, \dots, a_k, a_1', \dots, a_k' \in A$ we have 
\begin{equation}
	a_1 + \dots + a_k = a_1' + \dots + a_k' \quad \Leftrightarrow \quad f(a_1) + \dots + f(a_k) = f(a_1') + \dots + f(a_k').
\end{equation}
One can think of this as a generalization of a group isomorphism for which only operations of depth at most $k$ are required to be preserved. A subset $A$ of an arbitrary abelian group is said to be \emph{rectifiable of order $k$} if it is $F_k$-isomorphic to a set of integers. Note that we will generally only be interested in the case $k = 2$, where we will just use the term \emph{rectifiable}. Lastly, the \emph{additive dimension $\dim(A)$} of a set of integers $A \subset \ZZ$ is defined to be the largest $s \in \NN$ for which $A$ is $F_2$-isomorphic to some subset of $\ZZ^s$ that is not contained in a hyperplane.

In what follows we will usually use the notation $\mA$ to refer to sets in some cyclic group $\ZZ_m$ and the usual notation $A$ to refer to sets in the integers. Often $\mA$ will refer to the canonical projection from $\ZZ$ to some $\ZZ_m$ of some $A \subset \ZZ$.

\subsection{Covering Statements}

Deshouillers and Freiman stated the following result regarding covering properties of subsets of $\ZZ_m$ with very small sumset. Note that -- unlike the main statement this paper is interested in  -- this result concerns arbitrary $\ZZ_m$, that is, the integer $m$ does not have to be prime. This explains the weaker bounds and more complex statement.

\begin{theorem}[Deshouillers and Freiman~\cite{DF03}] \label{thm:DeshouillersFreiman}
	For any set $\mA \subset \ZZ_n$ satisfying $|2\mA| \leq 2.04|\mA|$ and $|\mA| \leq n/10^9$ there exists a proper subgroup $H < \ZZ_n$ such that the following holds:
	\begin{enumerate}
		\item	If $\mA$ is included in one coset of $H$ then $|\mA| > |H| / 10^9$.
		\item	If $\mA$ meets exactly $2$ or at least $4$ cosets of $H$ then it is included in an $\ell$-term arithmetic progression of cosets of $H$ where 
			\begin{equation}\label{eq:DF1}
				(\ell-1)|H| \leq |2\mA| - |\mA|.
			\end{equation}
		\item	If $\mA$ meets exactly three cosets of $H$ then it is included in an $\ell$-term arithmetic progression of cosets of $H$ where 
			\begin{equation}\label{eq:DF2}
				(\min(\ell,4)-1)|H| \leq |2\mA| - |\mA|.
			\end{equation}
	\end{enumerate}
	Furthermore, if $\ell \geq 2$ then there exists a coset of $H$ containing $\sfrac{2}{3} \, |H|$ elements from $\mA$.
\end{theorem}

We will also need the following straightforward observation in order to distinguish between integer sets of different additive dimension.

\begin{lemma} \label{lemma:higherdim}
	Let $A \subset \ZZ$ be given in normal form with $|A| \geq 3$ and $m > 1$ such that $m \mid \max(A)$. If the canonical projection of $A$ into $\ZZ_m$ is rectifiable, then $\dim(A) \geq 2$.
\end{lemma}

\begin{proof}
	Let $\varphi: \ZZ \to \ZZ_{m}$ denote the canonical projection. Note that $\{ (a, \varphi(a)) :  a \in A \} \subset \ZZ \times \ZZ_{m}$ is $F_2$-isomorphic to $A$, since for any $a_1,a_2,a_3,a_4 \in A$ we have $a_1 + a_2 = a_3 + a_4$ if and only if $(a_1,\varphi(a_1)) + (a_2,\varphi(a_2)) = (a_3,\varphi(a_3)) + (a_4,\varphi(a_4))$. As $\mA = \varphi(A)$ is rectifiable, there exists some $F_2$-isomorphism $f$ mapping $\mA$ into the integers. By the same argument as before, it follows that $\{ (a, \varphi(a)) :  a \in A \}$ and hence also $A$ is $F_2$-isomorphic to $\{ (a, f(\varphi(a))) :  a \in A \} \subset \ZZ^2$. We may without loss of generality assume that $f(0) = 0$ and note that since $A$ is in normal form and $|A| \geq 3$, there must exist some $a' \in A$ such that $\varphi(a') \neq 0$ and hence also $f(\varphi(a')) \neq 0$. Using the requirement that $m \mid \max(A)$, we observe that the three points $(0,f(\varphi(0))) = (0,0)$,  $(\max(A),f(\varphi(\max(A))) = (\max(A),0)$ and $(a',f(\varphi(a'))) \neq (a',0)$ do not lie in a hyperplane of $\ZZ^2$ and therefore $\dim(A) \geq 2$ as desired.
\end{proof}


We can now state and prove the main new ingredient needed for the proof of Theorem~\ref{thm:main}. It should be noted that the proof has some slight similarities with the proof of Freiman's $3|A|-4$ Theorem in the integers by modular reduction (see \cite{LS95}), but there is a new component in the argument here, consisting of taking into account the Freiman dimension of the set.

\begin{proposition} \label{prop:4k-8_partial}
	Any $1$-dimensional set $A \subset \ZZ$ satisfying $|2A| \leq 3.04 |A| - 3$ can be covered by an arithmetic progression of length at most $10^9 \, |A|$.
\end{proposition}

\begin{proof}
	Let $A \subset \ZZ$ satisfy $|2A| \leq 3.04 |A| - 3$ as well as $\max(A) \geq 10^9 |A|$, and assume without loss of generality that $A$ is in normal form. We will show that we must have $\dim(A) \geq 2$, in contradiction to the assumption that $A$ is $1$-dimensional. Let $\varphi : \ZZ \to \ZZ_{\max(A)}$ denote the canonical projection and observe that $\mA = \varphi(A)$ satisfies $|\mA| = |A| - 1 < \max(A)/10^9$. Let $B$ denote the set of elements $x \in 2A$ such that $x+\max(A)$ is also in $2A$.  Since $0$ and $\max(A)$ are both in $A$ we have $B\supset A$, whence $|2A| = |2\mA| + |B| \geq |2\mA| + |A|$, and so 
	\begin{align*}
		|2\mA| \leq |2A| - |A| \leq 2.04 |A| - 3 \leq 2.04 |\mA|.
	\end{align*}
	We can therefore apply Theorem~\ref{thm:DeshouillersFreiman}, obtaining that $\mA$ is covered by some small arithmetic progression of cosets of some proper subgroup $H < \ZZ_{\max(A)}$. Let us go through the cases given by this theorem. In the following $\mathds{1}_{B}$ will denote the indicator function of some given set $B$.
	\begin{enumerate}
		\item	As $A$ is in normal form, $\mA$ cannot be contained in a single coset of $H$.
		\item	If $\mA$ meets exactly $2$ or at least $4$ cosets of $H$ then it is included in an $\ell$-AP of cosets of $H$, where by \eqref{eq:DF1} we have $\ell \leq 1.04|\mA|/|H| + 1\leq (1.04+3/2)|\mA|/|H|$, the last equality following from the last sentence in Theorem \ref{thm:DeshouillersFreiman}. Using that $|\mA|<10^{-9}\max(A)$ we deduce that
			\begin{equation}
				\ell \leq 3|\mA|/|H| < \tfrac{1}{2} \, \left| \ZZ_{\max(A)/|H|} \right|.
			\end{equation}
			Letting $m = \max(A) / |H|$ we now observe that, since $A$ is in normal form, its canonical projection into $\ZZ_m$ cannot be contained in a proper subgroup of $\ZZ_m$. It follows that the common difference of the $\ell$-term arithmetic progression covering this projection of $A$ does not divide $m$, whence we can dilate by the inverse mod $m$ of this common difference, and it follows that the projection of $A$ is $F_2$-isomorphic to some subset of an interval of size $m/2$ in $\ZZ_m$. This projection is therefore rectifiable, so by Lemma~\ref{lemma:higherdim} we have $\dim(A) \geq 2$.
		\item	If $\mA$ meets exactly $3$ cosets of $H$, then we argue in a way similar to case 2, considering the projection of $A$ to $\ZZ_m$ where $m = \max(A)/|H|$. Here, however, we distinguish two cases, according to whether the 3 cosets are in arithmetic progression or not.
		
		 Assume that these cosets are in arithmetic progression with difference $d$. If we can rectify the 3-term progression formed by the cosets' representatives, then we can rectify the projection of $A$ into $\ZZ_m$. By applying Lemma~\ref{lemma:higherdim} as in case 2 we again obtain the contradiction $\dim(A) \geq 2$. If we cannot rectify the 3-term progression, then we must have $m < 6$, $d = m/3$ or $d = m/4$. We certainly have $m\geq 6$ since by \eqref{eq:DF2} we have $|H| \leq (|2\mA| - |\mA|)/2 \leq 0.52 |\mA|$, and as noted above we also have $\max(A) \geq 10^9 |\mA|$, so $m \geq 10^9$. Furthermore, if $d = m/3$ or $d = m/4$ then $m$ is multiple of $d$ and clearly $A$ cannot have been in normal form.
		 
		 If the cosets do not form an arithmetic progression, then we have $\mA \subseteq H \cup (H+c_1) \cup (H+c_2)$ for some $c_1,c_2 \in \ZZ_{\max(A)}$ satisfying $c_2 \not \equiv 2 c_1$, $c_1 \not \equiv 2 c_2$ and $c_1 + c_2 \not \equiv 0$ in $\ZZ_{\max(A)} / H$. Moreover, we may assume that either $2c_1 \not \equiv 0$ or $2c_2 \not \equiv 0$ in $\ZZ_{\max(A)} / H$ as otherwise $\mA$ would only meet $2$ cosets of $H$. We therefore assume without loss of generality that $2c_2 \not \equiv 0$ in $\ZZ_{\max(A)} / H$. If furthermore $2c_2 \not \equiv 2c_1$ in $\ZZ_{\max(A)} / H$, then $\{\mathds{1}_{H+c_2}(\varphi(a)) : a \in A \}$ is $F_2$-homomorphic to $A$ and therefore $\dim(A) \geq 2$ as $A$ is $F_2$-isomorphic to
			\begin{equation}
				\big\{ (\mathds{1}_{H+c_2}(\varphi(a)), a) : a \in A \big\} \subset \ZZ^2
			\end{equation}
			which is not contained in some hyperplane of $\ZZ^2$ as $\varphi(0) = \varphi(\max(A)) \in H$ but $\varphi(a') \in H + c_2$ for some $a' \in A$. If however $2c_1 \equiv 2c_2 \not \equiv 0$ in $\ZZ_{\max(A)} / H$, then likewise we can argue that $\dim(A) \geq 2$ as $A$ now is $F_2$-isomorphic to
			\begin{equation}
				\big\{ (\mathds{1}_{H}(\varphi(a)), a) : a \in A \big\} \subset \ZZ^2
			\end{equation}
			which for the same reason is also not contained in any hyperplane of $\ZZ^2$.
	\end{enumerate}
	It follows that $\dim(A) \geq 2$ in contradiction to the assumption that $A$ is $1$-dimensional.
\end{proof}

We will also need the following two results due to Freiman that will enable us to deal with sets past the $3|A|-4$ threshold that are not $1$-dimensional.

\begin{theorem}[Freiman] \label{thm:dimension_min}
	Every finite set $A \subset \ZZ$ of additive dimension $d$ satisfies
	\begin{equation}
		|2A| \geq (d+1)|A| - \binom{d+1}{2}.
	\end{equation}
\end{theorem}

For a proof see \cite[Lemma 1.14]{Freiman73}.

\begin{theorem}[Freiman] \label{thm:3.3k}
	Let $A \subset \ZZ^2$ be a $2$-dimensional set that cannot be embedded in any straight line and that satisfies $|2A| < \sfrac{10}{3} \, |A| - 5$ and $|A| \geq 11$. Then $A$ is contained in a set which is isomorphic to
	\begin{equation}
		\{(0,0),(0,1),(0,2),\dots,(0,k_1-1),(1,0),(2,0),\dots,(1,k_2-1)\}
	\end{equation}
	where $k_1,k_2 \geq 1$ and $k_1+k_2 \leq |2A|-2|A|+3$.
\end{theorem}

The above result is~\cite[Theorem~1.17]{Freiman73}. We shall use the following consequence.

\begin{corollary}\label{cor:Freiman_2lines}
	Any $2$-dimensional set $A \subset \ZZ$ satisfying $|2A| \leq \sfrac{10}{3} \, |A| - 7$ is contained in the union of two arithmetic progressions $P_1$ and $P_2$ with the same common difference such that $|P_1\cup P_2|\leq |2A| - 2|A| + 3$. Furthermore, the sumsets $2P_1$, $P_1+P_2$ and $2P_2$ are disjoint.
\end{corollary}

A proof of this can be immediately derived from the following statement.

\begin{lemma} \label{lemma:isomorphism}
	Given $d \geq 1$ and a finite set $A \subset \ZZ^d$ not contained in a hyperplane, we can extend any Freiman-isomorphism $\varphi$ mapping $A$ to some $A' \subset \ZZ$ to an affine linear map.
\end{lemma}

\begin{proof}
	Assume to the contrary that $\varphi$ is not affine linear. As $\dim(A') = \dim(A) = d$, there exist $d$ elements $a_1,\dots,a_d \in A'$ spanning $\ZZ^d$. Let $\varphi_e$ denote the affine linear map $\ZZ^d \to \ZZ$ determined by $a_1,\dots,a_d \in \ZZ^d$ as well as $0$, that is $\varphi_e(a_i) = \varphi(a_i)$ for $i = 1,\dots,d$ and $\varphi_e(0) = \varphi(0)$. As $\varphi$ is not affine linear, we must have $\varphi_e(x) \neq \varphi(x)$ for some $x \in A' \setminus \{a_1,\dots,a_d,0\}$. It follows that $A'' = \{(a,\varphi_e(a)-\varphi(a)) : a \in A'\} \subset \ZZ^{d+1}$ cannot be contained in a hyperplane, that is $\dim(A'') \geq d+1$. However, one can easily verify that $A''$ is Freiman-isomorphic to $A'$, giving us a contradiction.
\end{proof}

\begin{proof}[Proof of Corollary~\ref{cor:Freiman_2lines}]
	Let $A' \subset \ZZ^2$ denote a set that is $F_2$-isomorphic to $A$ and not contained in a line. By Theorem~\ref{thm:3.3k} we can assume that $A'$ is contained in two lines of combined size less than $|2A| - 2|A| + 3$. By Lemma~\ref{lemma:isomorphism} the $F_2$-isomorphism $\varphi$ mapping $A'$ to $A$ can be extended to an affine linear map, implying the desired statement.
\end{proof}

\subsection{The Fourier-analytic Rectification}

It is obvious that at least half of any set $\mA \subset \ZZ_p$ can be rectified. It is reasonable to expect that if $\mA$ is `concentrated' in some sense, then one should be able to rectify significantly more than just half of the set. Freiman stated such a result using the language of large Fourier coefficients. In the following $\widehat{\mathds{1}}_{\mA}(x)=\sum_{a\in \mA} e^{2\pi ax/p}$ will denote the Fourier transform of the indicator function of some set $\mA \subset \ZZ_p$. 

\begin{theorem}[Freiman~\cite{Freiman73}] \label{thm:FreimanRectification}
	For any $\mA \subset \ZZ_p$ and $d \in \ZZ_p^{\star}$ there exists $u \in \ZZ_p$ such that
	\begin{equation}
		\big| [u,u+p/2) \cap d \cdot \mA \big| \geq \frac{|\mA| + |\widehat{\mathds{1}}_{\mA}(d)|}{2}.
	\end{equation}
\end{theorem}

It should be noted that an improved version of this result can be obtained using a result of Lev~\cite{Lev05}. However, we will stick to using Theorem~\ref{thm:FreimanRectification} when proving our main statement, as the improvement that would follow from using Lev's result is negligible in our case. Lastly, it was also Freiman who noted that a small sumset implies the existence of a large Fourier coefficient and hence a certain `concentration' of the set. We state this observation in the following form. The proof follows by a standard application of the Cauchy-Schwarz inequality.

\begin{lemma}[Freiman~\cite{Freiman73}] \label{lemma:LargeFourierCoefficient}
	For any $\mA \subset \ZZ_p$ there exists $d \in \ZZ_p^{\star}$ such that 
	\begin{equation}
		\big|\widehat{\mathds{1}}_{\mA}(d)\big| \geq \left( \frac{p/|2\mA|-1}{p/|\mA|-1} \right)^{1/2} |\mA|.
	\end{equation}
\end{lemma}

\begin{proof}
	We start by observing that 
	\begin{align*}
		\sum_{a=0}^{p-1} \, \widehat{\mathds{1}}_{\mA}(a)^2 \, \overline{\widehat{\mathds{1}}_{2\mA}(a)} = \sum_{a=0}^{p-1} \,\sum_{x_1,x_2 \in \mA} \, \sum_{x_3 \in 2\mA} e^{2\pi i a (x_1 + x_2 -x_3)/p} = |\mA|^2p.
	\end{align*}
	Now if $|\widehat{\mathds{1}}_{\mA}(a)| \leq \theta \, |\mA|$ for all $a \neq 0 \mod p$ and
	\begin{equation*}
		\theta < \left( \frac{p/|2\mA|-1}{p/|\mA|-1} \right)^{1/2}
	\end{equation*}
	then using Cauchy-Schwarz one would get the contradiction
	\begin{align*}
		\sum_{a=0}^{p-1} \, \widehat{\mathds{1}}_{\mA}(a)^2 \, \overline{\widehat{\mathds{1}}_{2\mA}(a)} & = |\mA|^2 \, |2\mA| + \sum_{a=1}^{p-1} \, \widehat{\mathds{1}}_{\mA}(a)^2 \, \overline{\widehat{\mathds{1}}_{2\mA}(a)} \\
		& \leq |\mA|^2 \, |2\mA| + \theta |\mA| \left( \sum_{a=1}^{p-1} |\widehat{\mathds{1}}_{\mA}(a)|^2 \right)^{1/2}\left( \sum_{a=1}^{p-1} |\widehat{\mathds{1}}_{2\mA}(a)|^2  \right)^{1/2} \\
		& = |\mA|^2 \, |2\mA| + \theta |\mA| \left( |\mA|p - |\mA|^2 \right)^{1/2}\left( |2\mA| p - |2\mA|^2 \right)^{1/2} < |\mA|^2 p.
	\end{align*}
	The desired statement follows.
\end{proof}

\section{Proof of Theorem~\ref{thm:main}} \label{sec:proof}

Note that throughout the proof we will simplify notation by just writing $p/2$ and $p/3$ rather than the correct rounded version. In all cases there will be an appropriate amount of slack that justifies this simplification.

Let $d \in \ZZ_p^{\star}$ and $u \in \ZZ_p$ be such that $\mA_1 = [u,u+p/2) \cap d \cdot \mA$ satisfies
\begin{equation}
	|\mA_1| = \max_{u',d'} |[u',u'+p/2) \cap d' \cdot \mA|.
\end{equation}
We assume without loss of generality that $d = 1$ and $u = 0$. By Theorem~\ref{thm:FreimanRectification} and Lemma~\ref{lemma:LargeFourierCoefficient} we have that
\begin{equation} \label{eq:A1size}
|\mA_1| \geq \left( 1 + \left( \frac{p/|2\mA|-1}{p/|\mA|-1} \right)^{1/2} \right) \frac{|\mA|}{2} > 0.8175 \, |\mA|.
\end{equation}
We note that $\mA_1$ satisfies $|2\mA_1| \leq 3.04|\mA_1| - 7$ as otherwise we would get the contradiction
\begin{align} \label{eq:A1sumset}
2.48 \, |\mA| - 7 & \geq |2 \mA| \geq |2\mA_1| > 3.04|\mA_1| - 7 > 2.484 \, |\mA| - 7.
\end{align}
As $\mA_1$ is contained in an interval of size less than $p/2$, it is rectifiable and hence there exists some $F_2$-isomorphic set $A_1 \subset \ZZ$. We note that due to Theorem~\ref{thm:dimension_min} we have $\dim(A_1) \in \{1,2\}$. Let us distinguish between these two cases.

\medskip

\noindent {\bf Case 1.} If $\dim(A_1) = 1$, then by Proposition~\ref{prop:4k-8_partial} it is contained in an arithmetic progression of size less than $10^9|\mA_1|$. If the common difference $r$ of this progression is not 1, we may dilate by $r^{-1}\mod p$ and translate once more, so that we may assume that $\mA_1 \subset [0,10^9|\mA_1|]$. Since $|\mA_1|$ is by assumption the most elements any $p/2$-segment can contain of any dilate of $\mA$, it follows that $\mA \subset [0,10^9|\mA_1|] \cup [p/2,p/2 + 10^9|\mA_1|]$. Therefore $\big( 2 \cdot \mA \big) \subset [0, 2 \cdot 10^9 \, |\mA_1|] \subset  [0,p/2)$. Hence all of $\mA$ can be rectified, so the $3|A|-4$ statement in the integers gives the desired covering.

\medskip

\noindent {\bf Case 2.} If $\dim(A_1) = 2$ then we apply Corollary~\ref{cor:Freiman_2lines}, obtaining progressions $P_1,P_2$ with union covering $A_1$, with same common difference $r$. We claim that we can assume without loss of generality that $\mA_1\subset [0,3|\mA|)\cup [c,c+3|\mA|)\subset [0,p/2)$ with $0, c+3|\mA|-1 \in \mA$ for some $c \in \ZZ_p$  and $|\mA_1 \cap [0,3|\mA|) |\geq |\mA_1|/2$. Indeed, if $r\neq 1$, then we can dilate by $r^{-1}\mod p$ so as to ensure that $r^{-1}\cdot \mA_1 \subseteq [0,3|\mA|) \cup [c,c+3|\mA|)$ with $0, c+3|\mA|-1 \in r^{-1}\cdot \mA$. If $c+3|\mA| < p/2$, then the first two requirements are met and we can ensure that $|r^{-1}\cdot \mA_1 \cap [0,3|\mA|) |\geq |\mA_1|/2$ by multiplying the set with $-1$ and translating if necessary. If $p/2-3|\mA| \leq c \leq p/2+3|\mA|$, then $2\cdot r^{-1}\cdot \mA_1$ must lie in $[-6|\mA|,6|\mA|]$ and arguing as in case 1 we conclude that $2\cdot r^{-1}\cdot \mA\subset \{0,p/2\}+[-6|\mA|,6|\mA|]$, so $4\cdot r^{-1}\cdot  \mA\subset [-12|\mA|,12|\mA|]$ and again we can rectify all of $\mA$ and complete the argument this way. Lastly, if $p/2 + 3|\mA| < c$ then we simply translate the set by $-c$ to meet the first two requirements and again multiply by $-1$ if necessary. This proves our claim.

Now, let $\mS' = [0,3|\mA|)$, $\mS'' = [c,c+3|\mA|)$, $\mA_1' = \mA_1 \cap \mS'$ and $\mA_1'' = \mA_1 \cap \mS''$. By the claim above we have  $|\mA_1'| \geq |\mA_1|/2$ and $\mS' \cup \mS'' \subset [0,p/2)$. We now show that
\begin{equation} \label{eq:Rcontainement}
	\mR := \mA \setminus \mA_1 = \mA \setminus (\mS' \cup \mS'') \subset [2c-3|\mA|,2c+6|\mA|) = 2\mS'' + [-3|\mA|,0].
\end{equation}
We start by observing that by assumption $\mA_1 = \mA \cap (\mS' \cup \mS'')$ was the most of $\mA$ we could rectify. It follows that $[0,p/2) \setminus (\mS' \cup \mS'')$ does not contain any elements of $\mA$. Next, let us assume that there exists $a \in \mA$ satisfying $a \in [-3|\mA|,0)$. Since $\mA_1\cup \{a\}$ cannot be rectified, we must have $c+3|\mA| > p/2 - 3|\mA|$. This implies that $[c,c+3|\mA|)\subset [p/2-6|\mA|,p/2)$, whence
\begin{equation*}
	2 \cdot (\mA_1 \cup \{a\}) \subset [-12|\mA|,6|\mA|) \subset [0,p/2) -12|\mA|,
\end{equation*}
which contradicts our maximality assumption about $\mA_1$. It follows that $\mA \cap [-3|\mA|,0) = \emptyset$. Arguing similarly, we see that $\mA \cap [c+3|\mA|,c+6|\mA|) = \emptyset$: certainly $\mA \cap [c+3|\mA|,p/2)\neq \emptyset$, and if there is $a\in \mA \cap [p/2,c+6|\mA|)\subset [p/2,p/2+3|\mA|)$ then $[c,c+3|\mA|)\subset [p/2-6|\mA|,p/2)$, and so $2\cdot(\mA_1\cup \{a\})\subset [-12|\mA|,12|\mA|)$, again contradicting our maximality assumption.

Next, we note that
\begin{equation*}
	2\mA_1 \subset [0,6|\mA|) \cup [c,c+6|\mA|) \cup [2c,2c+6|\mA|).
\end{equation*}
It follows that if there exists $a  \in \mA$ satisfying
\begin{equation*}
	a \in [p/2,0) \setminus \big( [-3|\mA|,6|\mA|) \cup [c -3|\mA|,c+6|\mA|) \cup [2c -3|\mA|,2c+6|\mA|) \big),
\end{equation*}
then $a + \mA_1'$ does not intersect $2\mA_1$ and we get the contradiction
\begin{align*}
	|2\mA| \geq |2\mA_1| + |a + \mA_1'| & \geq (2|\mA_1'| - 1) + (2|\mA_1''| - 1)  + (|\mA_1'| + |\mA_1''| - 1) + |\mA_1'| \\
	& \geq 3.5|\mA_1| - 3 > 2.48|\mA| - 7.
\end{align*}
Note that we have used that $2\mA_1' \cap (A_1' + A_1'') = \emptyset$ as well as $2\mA_1'' \cap (A_1' + A_1'') = \emptyset$ as given by Corollary~\ref{cor:Freiman_2lines}. Using the previous observations, it follows that equation~\eqref{eq:Rcontainement} is established and we have $\mA_1 = \mA_1' \cup \mA_1'' \cup \mR$ where $\mR = \mA \cap [2c-3|\mA|,2c+6|\mA|)$. Note that we may assume that $|\mR| \geq 0.17|\mA|$ as otherwise $|\mA_1| \geq 0.83|\mA|$ and in equation~\eqref{eq:A1sumset} we would in fact get $|2\mA_1| \leq 3|A| - 4$, which due to Theorem~\ref{thm:dimension_min} would contradict our assumption that $\mA_1$ is $2$-dimensional.

We note that $2\mA \supseteq 2\mA_1 \cup (\mA_1'' + \mR)$ and that trivially $|\mA_1'' + \mR| \geq |\mR|$. It follows that $\mA_1'' + \mR$ must intersect $2\mA_1$ since otherwise we would get the contradiction
\begin{align*}
	|2\mA| \geq |2\mA_1| + |\mR| \geq 3.17|\mA_1| - 2 > 2.48|\mA| - 7.
\end{align*}

It follows that one of the following must hold:
\begin{itemize}
	\item[(i)]	If $(\mA_1''+\mR) \cap 2\mA_1'' \neq \emptyset$, then we must have 
	\begin{align*}
		3c + 9|\mA| - p \geq 2c \quad \text{and} \quad 3c - 3|\mA| - p \leq	 2c + 6|\mA|,
	\end{align*}
	and therefore $c \in [p - 9|\mA|,p + 9|\mA|]$. However, we know that $c \leq p/2$ and that the cardinality of $\mA$ is sufficiently small with respect to $p$, so we get a contradiction. it follows that 
	\item[(ii)]	If $(\mA_1''+\mR) \cap (\mA_1' + \mA_1'') \neq \emptyset$, then we must have 
	\begin{align*}
		3c + 9|\mA| - p \geq c \quad \text{and} \quad 3c - 3|\mA| - p \leq	 c + 6|\mA|,
	\end{align*}
	and therefore $c \in [p/2 - \sfrac{9}{2} \, |\mA|, p/2 + \sfrac{9}{2} \, |\mA|]$. Consequently, in this case $\mA_1'$ and $\mR$ are focused around $0$ and $\mA_1''$ is focused around $p/2$. It follows that a dilation by a factor of $2$ focuses all parts of $\mA$ around $0$, that is
	\begin{equation*}
		2 \cdot \mA \subset [-12|\mA|,15|\mA|) \subset -12|\mA| + [0,p/2).
	\end{equation*}
	This means that all of $\mA$ can be rectified and we can just apply the $3|\mA|-4$ statement in the integers to get the desired covering property.

	\begin{figure}[h]
	\begin{center}
	\bigskip
	\scalebox{0.8}{
	\begin{tikzpicture}
		\begin{scope}
			\draw[black, line width=0.3mm] (0,2.7) -- (0,3.3);
			\draw[black, line width=0.3mm] (0,-2.7) -- (0,-3.3);
			\draw[black, line width=0.3mm] (0,2.8) -- (0,3.2);
			\draw[black, line width=0.3mm] (0,2.8) -- (0,3.2);
			\draw[black, line width=0.3mm] (0,0) pic{carc=0:360:3};
			\draw[black, line width=2.5mm] (0,0) pic{carc=90:65:3};
			\draw[black, line width=2.5mm] (0,0) pic{carc=-20:-45:3};
			\draw[black, line width=2.5mm] (0,0) pic{carc=-140:-165:3};
			\node[scale=1.2, black] at (0,3.7) {$0$};		\node[scale=1.2, black] at (0,-3.7) {$p/2$};
			\node[scale=1, black] at (.6,2.4) {$\mA_1'$};
			\node[scale=1, black] at (2.1,-1.3) {$\mA_1''$};
			\node[scale=1, black] at (-2.3,-1.1) {$\mR$};
		\end{scope}
	\end{tikzpicture}
	\hspace{1cm}
	\begin{tikzpicture}
		\begin{scope}
			\draw[black!20, line width=3.6mm] (0,0) pic{carc=100:40:3};
			\draw[black!20, line width=3.6mm] (0,0) pic{carc=-10:-70:3};
			\draw[black!20, line width=3.6mm] (0,0) pic{carc=-140:-190:3};
			\draw[black, line width=2mm] (0,0) pic{carc=90:50:3};
			\draw[black, line width=2mm] (0,0) pic{carc=-20:-60:3};
			\draw[black, line width=2mm] (0,0) pic{carc=-150:-180:3};
			\node[scale=1.2, black] at (0,3.7) {$0$};		\node[scale=1.2, black] at (0,-3.7) {$p/2$};
			\node[scale=1, black] at (.7,2.3) {$2\mA_1'$};
			\node[scale=1, black] at (1.4,-1.6) {$\mA_1' + \mA_1''$};
			\node[scale=1, black] at (-2.3,-.7) {$2\mA_1''$};
			\node[scale=1, black!50] at (1.5,3.3) {$\mR + \mA_1''$};
			\node[scale=1, black!50] at (2.6,-2.4) {$2\mR$};
			\node[scale=1, black!50] at (-3.9,-.8) {$\mR + \mA_1'$};
			\draw[black, line width=0.3mm] (0,2.7) -- (0,3.3);
			\draw[black, line width=0.3mm] (0,-2.7) -- (0,-3.3);
			\draw[black, line width=0.3mm] (0,2.8) -- (0,3.2);
			\draw[black, line width=0.3mm] (0,2.8) -- (0,3.2);
			\draw[black, line width=0.3mm] (0,0) pic{carc=0:360:3};
		\end{scope}
	\end{tikzpicture}
	}
	\end{center}
	\caption{Distribution of $\mA$ and $2\mA$ in $\ZZ_p$ in case~(iii).} \label{fig:2dim}
	\end{figure}
		
	\item[(iii)]	If $(\mA_1''+\mR) \cap 2\mA_1' \neq \emptyset$, then we must have 
	\begin{align*}
		3c + 9|\mA| - p \geq 0 \quad \text{and} \quad 3c - 3|\mA| - p \leq	6|\mA|,
	\end{align*}
	and therefore $c \in [p/3 -3|\mA|,p/3+3|\mA|]$. Consequently, in this case $\mA_1',\mA_1''$ and $\mR$ (or rather the intervals containing them) are roughly "equally distributed" in $\ZZ_p$, that is they are respectively focused around $0,p/3$ and $2p/3$ as illustrated in Figure~\ref{fig:2dim}. It follows that a dilation by a factor of $3$ focuses all parts of $\mA$ around $0$, that is
	\begin{equation*}
		3 \cdot \mA \subset [-27|\mA|,54|\mA|) \subset -27|\mA| + [0,p/2).
	\end{equation*}
	This again means that all of $\mA$ can be rectified and we can just apply the $3|\mA|-4$ statement in the integers to get the desired covering property.
\end{itemize}

It follows that we have proven the statement of Theorem~\ref{thm:main}. \qed

\section{Concluding remarks} \label{sec:remarks}

It is probably unreasonable to expect that this rectification methodology (of rectifying a large part of the set and arguing from there) will lead to a proof of the full conjecture with tight constants. In fact, the more natural direction seems to be to apply covering results in the cyclic group in order to prove covering statements in the integers. Both Lev and Smeliansky's proof of Freiman's $3|A|-4$ statement in the integers as well as the proof of Proposition~\ref{prop:4k-8_partial} fall into that category.

Even if all other ingredients existed in their ideal form, the rectification argument through a large Fourier coefficient appears to imply an inherent loss in the density. This is a problem concerning not only Freiman's original approach and the result presented here, but also the broader result of Green and Ruzsa.

\bigskip

Two almost identical conjectures, differing only by a constant of $1$, have been made as to what the true form of a statement like Theorem~\ref{thm:main} should look like, see~\cite{HSZ05, G13, CR17} as well as~\cite{SZ09}. It is clear that any such statement should include the following result of Vosper.

\begin{theorem}[Vosper~\cite{Vosper56}]
	Let $\mA \subseteq \ZZ_p$ satisfy $|\mA| \geq 2$ and $|2\mA| \leq p-2$. Then $|2\mA| = 2|\mA| - 1$ if and only if $\mA$ is an arithmetic progression.
\end{theorem}

The conjecture stated in~\cite{SZ09} has the advantage of being such a generalization, but unfortunately it does not hold in its stated form. On the other hand, we believe the conjecture stated in~\cite{HSZ05, G13, CR17} to be true, but it does not include the result of Vosper. We therefore propose the following combined version of the conjecture.
\begin{conjecture} \label{conj:main}
	Let a set $\mA \subset \ZZ_p$ be given. If either
	\begin{enumerate}
		\item[(i)]	$0 \leq |2\mA| - \big(2|\mA| - 1\big) \leq \min(|\mA| - 4, p - |2\mA| - 2)$ or 
		\item[(ii)]	$0 \leq |2\mA| - \big(2|\mA| - 1\big) = |\mA| - 3 \leq p - |2\mA| - 3$
	\end{enumerate}
	then $\mA$ can be covered by an arithmetic progression of length at most $|2\mA| - |\mA| + 1$.
\end{conjecture}

Let us provide two examples which show that this statement, if true, is tight. The first example proves the need for the requirement $|2\mA| - \big(2|\mA| - 1\big) \leq p - |2\mA| - 2$ in case~(i) and the second example proves that the case $|2\mA| - \big(2|\mA| - 1\big) = |\mA| - 3$ needs to be handled separately.

\begin{example}[Serra and Z{\'e}mor~\cite{SZ09}]
	Given $k \geq 2$ and $0 \leq x \leq k-3$ let $p = 2k + 2x - 1$ be prime and $\mA = \{0\} \cup \{x+2,x+3,...,(p+1)/2\} \subset \ZZ_p$ so that $|\mA| = k$. We have $$2\mA = \{x+2,\dots,p-1,0,1\}= \ZZ_p \setminus \{2,\dots,x+1\}$$ so that $|2\mA| = 2|\mA| - 1 + x = p - x$. Clearly $\mA$ cannot be covered by an arithmetic progression of length at most $|\mA| + x = |2\mA| - |\mA| + 1$.
\end{example}

To see that this example not only implies the requirement  $|2\mA| - \big(2|\mA| - 1\big) \leq p - |2\mA| - 1$ but also  $|2\mA| - \big(2|\mA| - 1\big) \leq p - |2\mA| - 2$, note that $|2\mA| - \big(2|\mA| - 1\big) = p - |2\mA| - 1$ would imply that the prime $p$ is even.

\begin{example}
	Take a prime $p = 4t - 1$. Let $\mA = \{0,\dots,t\} \setminus \{t-1\} \cup \{2t\}$, that is $|\mA| = t+1$. We have $$2\mA = \{0,\dots,2t-2\} \cup \{2t,\dots,3t-2\} \cup \{3t\} = \ZZ_p \setminus \{2t-1,3t-1,3t+1,\dots, 4t-2\}$$ so that $|2\mA| = 3t - 1 = 3|\mA| - 4 = 2|\mA| - 1 + x = p - t = p - (x+2)$ where $x = |\mA| - 3$. Clearly $\mA$ cannot be covered by an arithmetic progression of length at most $|\mA| + x = |2\mA| - |\mA| + 1$.
\end{example}

Besides satisfying these two examples as well as implying the full strength of Vosper's Theorem in the case $|2\mA| - \big(2|\mA| - 1\big) = 0$, Conjecture~\ref{conj:main} would also imply the following part of a conjecture of Freiman~\cite{FS16} in the integers, giving perhaps some additional intuition as to it's slightly unusual shape.

\begin{corollary}
	Assume that Conjecture~\ref{conj:main} holds and let $A \subset \ZZ$ be a $1$-dimensional set in normal form for which $\max(A)$ is prime. If $|2A| = 3|A| - 4 + b \leq 4|A|-8$, then $A$ can be covered by an arithmetic progression of length at most $2(|A|+b-2) + 1$. If $|2A| = 4|A| - 7$, then $A$ can be covered by an arithmetic progression of length at most $4|A|-7$.
\end{corollary}

\begin{proof}
	Let $\mA$ denote the canonical embedding of $A$ into $\ZZ_{\max(A)}$. We start with the case $|2A| \leq 4|A|-8$ and note as in the proof of Proposition~\ref{prop:4k-8_partial} that $$|2\mA| \leq |2A| - |A| \leq 2|\mA| - 1 + (b-1) \leq 3|\mA| - 5.$$ Setting $x = |2\mA| - (2|\mA|-1) \leq b-1 \leq |\mA| - 4$ it would follow from case~(i) of Conjecture~\ref{conj:main} that either $|2\mA| > \max(A) - (x+2)$ and therefore we get the desired covering property for $A$, or that $\mA$ can be contained in an arithmetic progression of length at most $|\mA|+x \leq (\max(A)+1)/2$, implying that $\mA$ is rectifiable and therefore by Lemma~\ref{lemma:higherdim} contradicting the requirement that $A$ is $1$-dimensional.
	
	Now if $|2A| = 4|A| - 7$ then $x = |2\mA| - (2|\mA|-1) \leq b-1 \leq |\mA| - 3$ and we again either get the desired covering property from Conjecture~\ref{conj:main}, or we get a contradiction to the requirement that $A$ is $1$-dimensional.
\end{proof}

Note that this proof is essentially the same as that of Proposition~\ref{prop:4k-8_partial}. The requirement that $\max(A)$ is prime does not appear in Freiman's original conjecture and is artificial, but the corollary gives an indication of the relationship between Conjecture~\ref{conj:main} and the mentioned conjecture by Freiman. In fact, the bounds given in the statement would imply that $\max(A)$ is even in the extremal case. To prove such a statement without that condition one would require an analogue of Conjecture~\ref{conj:main} in general $\ZZ_n$, that is a strengthening of the results of Kemperman~\cite{K60} or Deshouillers and Freiman~\cite{DF03}. To our knowledge, such a conjecture has not been explicitly formulated and might in fact be very intricate to state.

\bibliography{bib}

\begin{thebibliography}{10}

\bibitem{CR17}
P.~Candela and A.~de~Roton.
\newblock On sets with small sumset in the circle.
\newblock {\em arXiv preprint arXiv:1709.04501}, 2017.

\bibitem{Chang02}
M.-C. Chang et~al.
\newblock A polynomial bound in {F}reiman's theorem.
\newblock {\em Duke mathematical journal}, 113(3):399--419, 2002.

\bibitem{DF03}
J.-M. Deshouillers and G.~A. Freiman.
\newblock A step beyond {K}neser's theorem for abelian finite groups.
\newblock {\em Proceedings of the London Mathematical Society}, 86(1):1--28,
  2003.

\bibitem{Freiman61}
G.~Freiman.
\newblock Inverse problems in additive number theory. {A}ddition of sets of
  residues modulo a prime.
\newblock In {\em Dokl. Akad. Nauk SSSR}, volume 141, pages 571--573, 1961.

\bibitem{Freiman73}
G.~Freiman.
\newblock Foundations of a structural theory of set addition. {T}ranslated from
  the {R}ussian. {T}ranslations of mathematical monographs.
\newblock {\em American Mathematical Society, Providence, RI}, 1973.

\bibitem{FS16}
G.~A. Freiman and O.~Serra.
\newblock On doubling and volume: Chains.
\newblock {\em arXiv preprint arXiv:1608.04916, to appear in Acta Arithmetica},
  2016.

\bibitem{GR06}
B.~Green and I.~Z. Ruzsa.
\newblock Sets with small sumset and rectification.
\newblock {\em Bulletin of the London Mathematical Society}, 38(1):43--52,
  2006.

\bibitem{GR07}
B.~Green and I.~Z. Ruzsa.
\newblock {F}reiman's theorem in an arbitrary abelian group.
\newblock {\em Journal of the London Mathematical Society}, 75(1):163--175,
  2007.

\bibitem{G13}
D.~J. Grynkiewicz.
\newblock {\em Structural additive theory}, volume~30.
\newblock Springer Science \& Business Media, 2013.

\bibitem{HSZ05}
Y.~O. Hamidoune, O.~Serra, and G.~Z{\'e}mor.
\newblock On the critical pair theory in z/pz.
\newblock {\em arXiv preprint math/0507561}, 2005.

\bibitem{Jin07}
R.~Jin.
\newblock Freiman's inverse problem with small doubling property.
\newblock {\em Advances in Mathematics}, 216(2):711--752, 2007.

\bibitem{K60}
J.~H.~B. Kemperman.
\newblock On small sumsets in an abelian group.
\newblock {\em Acta Mathematica}, 103(1-2):63--88, 1960.

\bibitem{Lev05}
V.~F. Lev.
\newblock Distribution of points on arcs.
\newblock {\em Integers}, 5(2):A11, 2005.

\bibitem{LS95}
V.~F. Lev and P.~Y. Smeliansky.
\newblock On addition of two distinct sets of integers.
\newblock {\em Acta Arithmetica}, 70(1):85--91, 1995.

\bibitem{Rodseth06}
{\O}.~J. R{\o}dseth.
\newblock On {F}reiman's 2.4-theorem.
\newblock {\em Skr. K. Nor. Vidensk. Selsk}, (4):11--18, 2006.

\bibitem{Ruzsa94}
I.~Z. Ruzsa.
\newblock Generalized arithmetical progressions and sumsets.
\newblock {\em Acta Mathematica Hungarica}, 65(4):379--388, 1994.

\bibitem{Sanders08}
T.~Sanders.
\newblock Appendix to {R}oth's theorem on progressions revisited by {J}.
  {B}ourgain.
\newblock {\em Journal d'Analyse Math{\'e}matique}, 104(1):193--206, 2008.

\bibitem{Schoen11}
T.~Schoen.
\newblock Near optimal bounds in {F}reiman's theorem.
\newblock {\em Duke Mathematical Journal}, 158(1):1--12, 2011.

\bibitem{SZ09}
O.~Serra and G.~Z{\'e}mor.
\newblock Large sets with small doubling modulo $ p $ are well covered by an
  arithmetic progression.
\newblock In {\em Annales de l'institut Fourier}, volume~59, pages 2043--2060,
  2009.

\bibitem{Vosper56}
A.~G. Vosper.
\newblock The critical pairs of subsets of a group of prime order.
\newblock {\em Journal of the London Mathematical Society}, 1(2):200--205,
  1956.

\end{thebibliography}
\bibliographystyle{abbrv}

\end{document}